\documentclass[final,oneeqnum]{siamltex}
\usepackage{amsfonts, amssymb,amsmath}
\usepackage[active]{srcltx}
\title{Block triangular miniversal
deformations of
matrices and matrix pencils\footnote{This paper was published in: V. Olshevsky, E. Tyrtyshnikov (Eds), Matrix Methods: Theory, Algorithms and Applications, World Scientific Publishing Co. Pte. Ltd., Hackensack, NJ, 2010, pp. 69-84.}}

\author{
Lena Klimenko\thanks{Information and Computer Centre of the Ministry of Labour and Social Policy of Ukraine, Esplanadnaya 8/10, Kiev, Ukraine (e.n.klimenko@gmail.com).}
   \and
Vladimir V. Sergeichuk\thanks{Institute of Mathematics,
Tereshchenkivska 3, Kiev,
Ukraine (sergeich@imath.kiev.ua).
}}

\newtheorem{remark}[theorem]{Remark}
\newcommand{\ddd}{
\text{\begin{picture}(12,8)
\put(-2,-4){$\cdot$}
\put(3,0){$\cdot$}
\put(8,4){$\cdot$}
\end{picture}}}

\renewcommand{\le}{\leqslant}
\renewcommand{\ge}{\geqslant}

\begin{document}
\maketitle

\begin{abstract}
For each square complex matrix, V. I.
Arnold
constructed a
normal form with the minimal number of parameters to
which a family of all
matrices $B$ that are close enough to
this matrix can be reduced by
similarity
transformations that
smoothly depend on the
entries of $B$.
Analogous normal forms were also
constructed for families of
complex matrix pencils by A. Edelman, E.
Elmroth, and B.
K\r{a}gstr\"{o}m, and contragredient matrix pencils (i.e., of matrix pairs up to transformations $(A,B)\mapsto(S^{-1}AR,R^{-1}BS)$)  by M. I. Garcia-Planas and
V. V. Sergeichuk. In this paper we give other normal forms for families of matrices, matrix pencils, and contragredient matrix pencils; our normal forms are block triangular.
\end{abstract}



\section{Introduction}
\label{introd}

The reduction of a
matrix to its Jordan
form is an unstable
operation: both the
Jordan form and the
reduction
transformations depend
discontinuously on the
entries of the
original matrix.
Therefore, if the
entries of a matrix
are known only
approximately, then it
is unwise to reduce it
to Jordan form.
Furthermore, when
investigating a family
of matrices smoothly
depending on
parameters, then
although each
individual matrix can
be reduced to its Jordan
form, it is unwise to
do so since in such an
operation the
smoothness relative to
the parameters is
lost.

For these reasons,
Arnold \cite{arn}
constructed a \emph{miniversal
deformation} of any Jordan canonical matrix $J$; that is, a
family of matrices in a neighborhood of $J$
with the minimal number of parameters, to
which all
matrices $M$ close to
$J$ can be reduced by
similarity
transformations that
smoothly depend on the
entries of $M$ (see Definition \ref{ite}).

Miniversal
deformations were also
constructed for:
\begin{itemize}
  \item[(i)]
the Kronecker canonical form of
complex matrix pencils
by Edelman, Elmroth,
and K\r{a}gstr\"{o}m
\cite{kag}; another\
miniversal
deformation  (which is \emph{simple} in the sense of Definition \ref{fruo}) was
constructed by Garcia-Planas and
Sergeichuk
\cite{gar_ser};

  \item[(ii)]
the Dobrovol'skaya and Ponomarev  canonical form of complex
contragredient matrix
pencils (i.e., of
matrices of
counter linear
operators
$U\rightleftarrows V$)
in \cite{gar_ser}.
\end{itemize}

Belitskii \cite{bel_rus} proved that each Jordan canonical matrix $J$ is permutationally similar to some matrix $J^{\#}$, which is called a \emph{Weyr canonical matrix} and possesses the property: all matrices that commute with $J^{\#}$ are block triangular. Due to this property, $J^{\#}$ plays a central role in Belitskii's algorithm for reducing the matrices of any system of linear mappings to canonical form, see \cite{bel_engl,ser_can}.

In this paper, we find another property of Weyr canonical matrices: they possess block triangular miniversal deformations (in the sense of Definition \ref{fruo}).
Therefore, if we consider, up to smooth similarity transformations, a family of matrices that are close enough to a given square matrix, then we can take it in its Weyr canonical form $J^{\#}$ and the family in the form $J^{\#}+E$, in which $E$ is block triangular.

We also give block triangular miniversal deformations
of those canonical forms of pencils and contragredient pencils that are obtained from (i) and (ii) by replacing the Jordan canonical matrices with the Weyr canonical matrices.

All matrices that we consider are complex matrices.

\section{Miniversal deformations of matrices}\hfil
\label{s2}

\begin{definition}[see
\cite{arn,arn2,arn3}]\label{ite} A \emph{deformation}
of an $n$-by-$n$ matrix $A$ is a matrix function
${\cal A}(\alpha_1, \dots,\alpha _k)$
$($its arguments $\alpha_1, \dots,\alpha _k$ are called \emph{parameters}$)$ on a neighborhood of $\vec 0=(0,\dots,0)$ that is holomorphic at $\vec 0$ and equals $A$  at $\vec 0$.
Two deformations of $A$ are identified if they coincide on a neighborhood of $\vec 0$.

A deformation ${\cal A}(\alpha_1, \dots,\alpha _k)$ of $A$ is \emph{versal} if all matrices $A+E$ in some neighborhood of $A$ reduce to the form
\[
{\cal A}(h_1(E),\dots,h_k(E))={\cal S}(E)^{-1}(A+E){\cal S}(E),\qquad S(0)=I_n,
\]
in which ${\cal S}(E)$ is a holomorphic at zero matrix function of the entries of $E$.

A versal deformation with the minimal number of parameters is called \emph{miniversal}.
\end{definition}

\begin{definition}\label{fruo}
Let a deformation ${\cal
A}$ of $A$ be represented in the form $A+{\cal
B}(\alpha_1, \dots,\alpha _k)$.
\begin{itemize}
  \item
If $k$ entries of ${\cal
B}(\alpha_1, \dots,\alpha _k)$ are the independent parameters $\alpha _1,\dots,\alpha _k$ and the others are zero then the deformation ${\cal
A}$ is called \emph{simple}\footnote{Arnold's miniversal definitions presented in Theorem \ref{teo2} are simple. Moreover, by \cite[Corollary 2.1]{gar_ser} the set of matrices of any quiver representation (i.e., of any finite system of linear mappings) over $\mathbb C$ or $\mathbb R$ possesses a simple miniversal deformation.}.

  \item
A simple deformation is \emph{block triangular} with respect to some partition of $A$ into blocks if ${\cal
B}(\alpha_1, \dots,\alpha _k)$ is block triangular with respect to the conformal partition and each of its blocks is either $0$ or all of its entries are independent parameters.

\end{itemize}
\end{definition}

If ${\cal
A}(\alpha_1, \dots,\alpha _k)$ is a miniversal deformation of $A$ and $S^{-1}AS=B$ for some nonsingular $S$, then $S^{-1}{\cal
A}(\alpha_1, \dots,\alpha _k)S$ is a miniversal deformation of $B$. Therefore, it suffices to construct miniversal deformations of canonical matrices for similarity.

Let
\begin{equation}\label{kid}
J(\lambda ):=J_{n_1}(\lambda)\oplus\dots \oplus J_{n_l}(\lambda),\qquad n_1\ge n_2\ge\dots\ge n_l,
\end{equation}
be a Jordan canonical matrix with a single eigenvalue equal to $\lambda$; the unites of Jordan blocks are written over the diagonal:
\[
J_{n_i}(\lambda):=\begin{bmatrix}
\lambda&1&&0\\&\lambda&\ddots&\\&&\ddots&1
\\ 0&&&\lambda
\end{bmatrix}\qquad \text{($n_i$-by-$n_i$)}.
\]
For each natural numbers $p$ and $q$, define the $p\times q$ matrix
\begin{equation}       \label{5.1}
{\cal T}_{pq}:=
  \begin{cases}
    \begin{bmatrix}
         * & 0&\dots&0 \\ \vdots&\vdots&&\vdots \\
         * & 0&\dots&0 \\
       \end{bmatrix}
& \text{if $p< q$}, \\[7mm]
    \begin{bmatrix}
         0&\dots&0 \\ \vdots&&\vdots \\
         0&\dots&0 \\
         * &\dots&* \\
       \end{bmatrix} & \text{if $p\ge q$},
  \end{cases}
\end{equation}
in which the stars denote
independent parameters (alternatively, we may take ${\cal T}_{pq}$ with $p=q$ as in the case $p<q$).

\begin{theorem}[{\cite[\S 30, Theorem 2]{arn3}}]
\label{teo2}
{\rm(i)}
Let $J(\lambda )$ be a Jordan
canonical matrix of the form  \eqref{kid} with a single eigenvalue equal to $\lambda$. Let ${\cal H}:=[{\cal T}_{n_i,n_j}]$ be the parameter block matrix partitioned conformally to $J(\lambda )$ with the blocks ${\cal T}_{n_i,n_j}$ defined in \eqref{5.1}.  Then
\begin{equation}\label{lir}
 J(\lambda )+{\cal H}
\end{equation}
is a simple miniversal deformation of $J(\lambda )$.

{\rm(ii)}  Let
\begin{equation}\label{juyt}
J:=J(\lambda_1)\oplus\dots\oplus J(\lambda_ {\tau}),\qquad \text{$\lambda_i\ne \lambda_j$ if $i\ne j$},
\end{equation}
be a Jordan canonical matrix in which every $J(\lambda_i)$ is of the form \eqref{kid}, and let $J(\lambda_i)+{\cal H}_{i}$ be its miniversal deformation \eqref{lir}. Then
\begin{equation}\label{kuc}
J+{\cal K}:=(J(\lambda_1)+{\cal H}_{1})\oplus\dots\oplus (J(\lambda_{\tau})+{\cal H}_{{\tau}})
\end{equation}
is a simple miniversal deformation of $J$.
\end{theorem}

\begin{definition}[\cite{wey}]\label{jyf}
The \emph{Weyr canonical form} $J^{\#}$ of a Jordan canonical matrix $J$ $($and of any matrix that is similar to $J)$ is defined as follows.

{\rm(i)} If $J$ has a single eigenvalue, then we write it in the form \eqref{kid}.
Permute the first columns of $J_{n_{1}}(\lambda )$, $J_{n_{2}}(\lambda )$,
\dots, and $J_{n_{l}}(\lambda )$ into the first $l$ columns, then permute the
corresponding rows.  Next permute the second columns of all blocks of size at least $2\times 2$ into the next columns and permute the
corresponding rows; and so on. The obtained matrix is the Weyr canonical form $J(\lambda )^{\#}$ of $J(\lambda )$.

{\rm(ii)} If $J$ has distinct eigenvalues, then we write it in the form \eqref{juyt}. The Weyr canonical form of $J$ is
\begin{equation}\label{dul}
J^{\#}:=J(\lambda_1)^{\#} \oplus\dots\oplus J(\lambda_ {\tau})^{\#}.
\end{equation}
\end{definition}

Each direct
summand of \eqref{dul} has the form
\begin{equation}\label{utk}
J(\lambda )^{\#}=\begin{bmatrix}
\lambda I_{s_{1}}&\begin{bmatrix}
I_{s_{2}}\\0
\end{bmatrix}&&0\\
&\lambda I_{s_{2}}&\ddots&\\
&&\ddots&\begin{bmatrix}
I_{s_{k}}\\0
\end{bmatrix}\\ 0&&&\lambda I_{s_{k}}
\end{bmatrix},
\end{equation}
in which $s_i$ is the number of Jordan
blocks $J_l(\lambda )$ of size $l\ge i$ in $J(\lambda )$. The sequence
$(s_{1}, s_{2},\dots,s_{k})$ is called the \emph{Weyr
characteristic} of $J$  (and of any matrix that is similar to $J$) for the eigenvalue $\lambda$, see
\cite{sha}.
By \cite{bel_rus} or \cite[Theorem 1.2]{ser_can}, all matrices commuting with $J^{\#}$ are block triangular.

In the next lemma we construct a miniversal deformation of $J^{\#}$ that is block triangular with respect to the most coarse
partition of $J^{\#}$ for which all diagonal blocks have the
form $\lambda_i I$ and each off-diagonal block
is $0$ or $I$. This means that the sizes of diagonal blocks of
\eqref{utk} with respect to this partition form the sequence obtained from
\[
\begin{array}{l}
  s_k,\ s_{k-1}-s_k,\ \dots,\ s_2-s_3,\ s_1-s_2,\\
  s_k,\ s_{k-1}-s_k,\ \dots,\ s_2-s_3,\\ \dots\dots\dots\dots\dots\dots\\
  s_k,\ s_{k-1}-s_k,\\
  s_k
  \end{array}
\]
by removing the zero members.

\begin{theorem}      \label{t2h}
{\rm(i)}
Let $J(\lambda )$ be a Jordan
canonical matrix of the form  \eqref{kid} with a single eigenvalue equal to $\lambda$. Let $J(\lambda )+\cal H$ be its miniversal deformation \eqref{lir}. Denote by
\begin{equation}\label{kut}
J(\lambda )^{\#}+{\cal H}^{\#}
\end{equation}
the parameter matrix obtained from $J(\lambda )+{\cal H}$ by the permutations described in Definition {\rm\ref{jyf}(i)}. Then $J(\lambda )^{\#}+{\cal H}^{\#}$ is a miniversal deformation of $J(\lambda)^{\#}$ and its matrix ${\cal H}^{\#}$ is lower block triangular.

{\rm(ii)} Let $J$ be a Jordan canonical matrix represented in the form \eqref{juyt} and let $J^{\#}$ be its Weyr canonical form. Let us apply the permutations described in {\rm(i)} to each of the direct summands of miniversal deformation \eqref{kuc} of $J$. Then the obtained matrix \begin{equation}\label{kucj}
J^{\#}+{\cal K}^{\#}:=(J(\lambda_1)^{\#}+{\cal H}_1^{\#})\oplus\dots\oplus (J(\lambda_{\tau} )^{\#}+{\cal H}_{\tau}^{\#})
\end{equation}
is a miniversal deformation
of $J^{\#}$, which is simple and block triangular $($in the sense of Definition {\rm \ref{fruo})}.
\end{theorem}

Let us prove this theorem. The form of $J(\lambda )^{\#}+{\cal H}^{\#}$ and the block triangularity of ${\cal H}^{\#}$ become clearer if we
carry out the permutations from Definition \ref{jyf}(i) in two steps.

\emph{First step.}
Let us write the sequence $n_1,n_2,\dots,n_l$ from \eqref{kid} in the form
\[
\underbrace{m_1,\dots,m_1}
_{\mbox{$r_1$
times}}, \underbrace{m_2,\dots,m_2} _{\mbox{$r_2$ times}},\dots,
\underbrace{m_t,\dots,m_t} _{\mbox{$r_t$ times}},
\]
where
\begin{equation}\label{jte}
 m_1> m_2>\dots> m_t.
\end{equation}
Partition $J(\lambda )$ into $t$ horizontal and $t$ vertical strips of sizes
\[
 r_1m_1,\ r_2m_2,\ \dots,\ r_tm_t
 \]
(each of them contains Jordan blocks of the same size), produce the described permutations within each of  these strips, and obtain
\begin{equation}\label{kids}
J(\lambda )^{\text{\it +}} :=J_{m_1}(\lambda I_{r_1})\oplus\dots \oplus J_{m_t}(\lambda I_{r_t}),
\end{equation}
in which
\[
J_{m_i}(\lambda I_{r_i}):=\begin{bmatrix}
\lambda I_{r_i} &I_{r_i}&&0\\&\lambda I_{r_i}&\ddots&\\&&\ddots&I_{r_i}
\\ 0&&&\lambda I_{r_i}
\end{bmatrix}  \qquad \text{($m_i$ diagonal blocks)}.
\]
By the same permutations of rows and columns of $J(\lambda )+{\cal H}$, reduce ${\cal H}$ to
\[{\cal H}^{\text{\it +}} :=[\tilde {\cal T}_{m_i,m_j}(r_i,r_j)],\] in which every $\tilde {\cal T}_{m_i,m_j}(r_i,r_j)$ is obtained from the matrix ${\cal T}_{m_i,m_j}$ defined in \eqref{5.1} by replacing each entry $0$ with the $r_i\times r_j$ zero block and each entry $*$ with the $r_i\times r_j$ block
\begin{equation}\label{liy}
\star:=\begin{bmatrix}
        * &\dots&* \\
        \vdots &&\vdots \\
        * &\dots&* \\
      \end{bmatrix}.
\end{equation}

For example, if
\begin{equation}\label{jut}
J(\lambda)= \underbrace{J_4(\lambda)\oplus
\dots \oplus J_4(\lambda)}_{\mbox{$p$
times}}\oplus \underbrace{J_2(\lambda)\oplus
\dots \oplus J_2(\lambda)}_{\mbox{$q$ times}}
\end{equation}
then
\begin{equation}\label{iod}
J(\lambda)^{\text{\it
+}}= J_4(\lambda I_p) \oplus J_2(\lambda I_q)=\begin{matrix}
{\arraycolsep 0.40em
\begin{matrix}
\scriptstyle(1,1)&
\scriptstyle(1,2)&
\scriptstyle(1,3)&
\scriptstyle(1,4)&
\scriptstyle(2,1)&
\scriptstyle(2,2)
   \end{matrix}}&\\
\left[\begin{array}{cccc|cc}
  \lambda I_p & I_p&0&0&0&0 \\
  0& \lambda I_p & I_p&0&0&0 \\
  0&0&\lambda I_p & I_p&0&0 \\
  0&0&0&\lambda I_p &0&0 \\\hline
0&0&0&0&\lambda I_q & I_q \\
0&0&0&0&0&\lambda I_q \\
  \end{array}\right]
  &\!\!\!\!\begin{matrix}
\scriptstyle(1,1)\\
\scriptstyle(1,2)\\
\scriptstyle(1,3)\\
\scriptstyle(1,4)\\
\scriptstyle(2,1)\\
\scriptstyle(2,2)
   \end{matrix}
    \end{matrix}
\end{equation}
A strip is indexed by $(i,j)$ if it contains the $j$-th strip of $J_{m_i}(\lambda I_{r_i})$. Correspondingly,
\begin{equation}\label{kidh}
{\cal H}^{\text{\it +}}=\begin{matrix}
  {\arraycolsep 0.20em
  \begin{array}{cccccc}
  \scriptstyle(1,1)\,&
\scriptstyle(1,2)\,&
\scriptstyle(1,3)\,&
\scriptstyle(1,4)\,&
\scriptstyle(2,1)\,&
\scriptstyle(2,2)
   \end{array}}&\\
\left[ 
\begin{array}{cccc|cc}
   \ 0 \ & \ 0 \ & \ 0 \ & \ 0\  &\  0\  & \  0  \ \\
  0&0&0&0&0&0\\
  0&0&0&0&0&0 \\
  \star & \star & \star & \star & \star & \star \\\hline
\star&0&0&0&0&0 \\
\star&0&0&0&\star&\star \\
  \end{array}\right]
  &\!\!\!\!\begin{matrix}
\scriptstyle(1,1)\\
\scriptstyle(1,2)\\
\scriptstyle(1,3)\\
\scriptstyle(1,4)\\
\scriptstyle(2,1)\\
\scriptstyle(2,2)
   \end{matrix}
    \end{matrix}
\end{equation}

\emph{Second step.}
We permute in $J(\lambda)^+$ the first vertical strips of
\[
J_{m_1}(\lambda I_{r_1}),\, J_{m_2}(\lambda I_{r_2}),\,
\dots,\,J_{m_t}(\lambda I_{r_t})
\]
into the first $t$ vertical strips and permute the
corresponding horizontal strips, then permute the second vertical strips into the next vertical strips and permute the
corresponding horizontal strips; continue the process until $J(\lambda)^{\#}$ is achieved. The same permutations transform ${\cal H}^{\text{\it +}}$ to
${\cal H}^{\#}$.

For example, applying there permutations to \eqref{iod} and \eqref{kidh}, we obtain
\begin{equation}\label{lih}
J(\lambda)^{\#}=\begin{matrix}
{\arraycolsep 0.40em
\begin{matrix}
\scriptstyle(1,1)&
\scriptstyle(2,1)&
\scriptstyle(1,2)&
\scriptstyle(2,2)&
\scriptstyle(1,3)&
\scriptstyle(1,4)
   \end{matrix}}&\\
\left[\begin{array}{cc|cc|c|c}
  \lambda I_p & 0&I_p&0&0&0 \\
  0& \lambda I_q &0&I_q&0&0 \\\hline
  0&0&\lambda I_p & 0&I_p&0 \\
  0&0&0&\lambda I_q &0&0 \\\hline
0&0&0&0&\lambda I_p & I_p \\
\hline
0&0&0&0&0&\lambda I_p \\
  \end{array}\right]
  &\!\!\!\!\begin{matrix}
\scriptstyle(1,1)\\[0,1mm]
\scriptstyle(2,1)\\[0,1mm]
\scriptstyle(1,2)\\[0,1mm]
\scriptstyle(2,2)\\[0,1mm]
\scriptstyle(1,3)\\[0,1mm]
\scriptstyle(1,4)   \end{matrix}
    \end{matrix}
\end{equation}
and
\begin{equation}\label{oih}
{\cal H}^{\#}=\begin{matrix}
    {\arraycolsep 0.20em
  \begin{array}{cccccc}
\scriptstyle(1,1)\,&
\scriptstyle(2,1)\,&
\scriptstyle(1,2)\,&
\scriptstyle(2,2)\,&
\scriptstyle(1,3)\,&
\scriptstyle(1,4)
   \end{array}}&\\
\left[ 
\begin{array}{cc|cc|c|c}
   \ 0 \ & \ 0 \ & \ 0 \ & \ 0\  &\  0\  & \  0  \ \\
 \star&0&0&0&0&0\\ \hline
  0&0&0&0&0&0 \\
  \star & \star & 0 & \star & 0 & 0 \\\hline
0&0&0&0&0&0 \\\hline
\star & \star & \star & \star & \star & \star \\
  \end{array}\right]
  &\!\!\!\!\begin{matrix}
\scriptstyle(1,1)\\[0,1mm]
\scriptstyle(2,1)\\[0,1mm]
\scriptstyle(1,2)\\[0,1mm]
\scriptstyle(2,2)\\[0,1mm]
\scriptstyle(1,3)\\[0,1mm]
\scriptstyle(1,4)
   \end{matrix}
    \end{matrix}
\end{equation}

{\em Proof of Theorem \ref{t2h}}.
(i)
Following \eqref{iod}, we index the vertical (horizontal) strips of $J(\lambda)^{\text{\it
+}}$ in \eqref{kids} by the pairs of natural numbers
as follows: a strip is indexed by $(i,j)$ if it contains the $j$-th strip of $J_{m_i}(\lambda I_{r_i})$.
The pairs that index the strips of $J(\lambda)^{\text{\it
+}}$ form the sequence
\begin{equation}\label{hik}
\begin{array}{l}
(1,1),\ (1,2),\ \dots,\ (1,m_t),\ \dots,\ (1,m_2),\
\dots,\ (1,m_1),\\
(2,1),\ (2,2),\ \dots,\ (2,m_t),\ \dots,\ (2,m_2),\\
\cdots\cdots\cdots\cdots\cdots
\cdots\cdots\cdots\cdots\cdots\\
(t,1),\ \, (t,2),\ \dots,\ \,(t,m_t),
\end{array}
\end{equation}
which is is ordered lexicographically.
Rearranging the pairs by the columns of \eqref{hik}:
\begin{equation}\label{hik2}
(1,1),\ (2,1),\ \dots,\ (t,1);\ \dots;\ (1,m_t),\ (2,m_t),\ \dots,\ (t,m_t);\
\dots;\ (1,m_1)
\end{equation}
(i.e., as in lexicographic ordering but starting from the second elements of the pairs)
and making the same permutation of the corresponding strips in $J(\lambda)^{\text{\it
+}}$ and ${\cal H}^{\text{\it
+}}$, we obtain $J(\lambda)^{\#}$ and ${\cal H}^{\#}$; see examples \eqref{lih} and \eqref{oih}.

The $((i,j),(i',j'))$-th entry of ${\cal H}^{\text{\it +}}$ is a star if and only if
\begin{equation}\label{kud}
\text{either $i\le i'$ and $j=m_i$,
or $i>i'$ and $j'=1$.}
\end{equation}
By \eqref{jte}, in these cases $j\ge j'$ and if $j=j'$ then either $j=j'=m_i$ and $i=i'$, or $j=j'=1$ and $i>i'$. Therefore, ${\cal H}^{\#}$ is lower block triangular.

(ii) This statement follows from (i) and Theorem \ref{teo2}(ii).
\qquad\endproof

\begin{remark}\label{yr}
Let $J(\lambda)$ be a Jordan matrix with a single eigenvalue, let $m_1> m_2>\dots> m_t$ be the distinct sizes of its Jordan blocks, and let $r_i$ be the number of Jordan blocks of size $m_i$. Then the deformation $J(\lambda )^{\#}+{\cal H}^{\#}$
from Theorem {\rm\ref{t2h}} can be formally constructed as follows:
\begin{itemize}
  \item
$J(\lambda )^{\#}$ and ${\cal H}^{\#}$ are matrices of the same size; they are conformally  partitioned into horizontal and vertical strips, which are indexed by the pairs \eqref{hik2}.

  \item
The $((i,j),(i,j))$-th
diagonal block of $J(\lambda )^{\#}$ is $\lambda I_{r_i}$, its $((i,j),(i,j+1))$-th block is $I_{r_i}$, and its other blocks are zero.

  \item
The $((i,j),(i',j'))$-th block of ${\cal H}^{\text{\it +}}$ has the form \eqref{liy} if and only if \eqref{kud} holds; its other blocks are zero.  \end{itemize}
\end{remark}

\section{Miniversal deformations of matrix pencils}
\label{s3}

By Kronecker's theorem on matrix pencils (see \cite[Sect.
XII, \S 4]{gan}), each pair of $m\times n$ matrices reduces by equivalence transformations
\[
(A,B)\mapsto
(S^{-1}AR,S^{-1}BR),\quad\text{
$S$ and $R$ are nonsingular,}
\]
to a \emph{Kronecker canonical pair} $(A_{\text{kr}},B_{\text{kr}})$ being a direct sum, uniquely determined up to permutation of summands, of pairs of the form
\[
(I_r,J_r(\lambda)),\
(J_r(0),I_r),\ (F_r,G_r),\
(F_r^T,G_r^T),
\]
in which $\lambda\in {\mathbb
C}$ and
\begin{equation}       \label{3.1o}
F_r:=\begin{bmatrix}
         1&&0\\
         0&\ddots&\\
         &\ddots&1\\
         0&&0
         \end{bmatrix},\qquad
G_r:=\begin{bmatrix}
         0&&0\\
         1&\ddots&\\
         &\ddots&0\\
         0&&1
         \end{bmatrix}
         \end{equation}
are matrices of size  $r\times (r-1)$ with $r\ge
1$.

Definitions \ref{ite} and \ref{fruo} are extended to matrix pairs in a natural way.

Miniversal deformations of $(A_{\text{kr}},B_{\text{kr}})$ were obtained in \cite{kag,gar_ser}. The deformation obtained in \cite{gar_ser} is simple; in this section we reduce it to block triangular form by permutations of rows and columns. For this purpose, we replace in $(A_{\text{kr}},B_{\text{kr}})$
\begin{itemize}
  \item
the direct sum $(I,J)$ of all pairs of the form $(I_r,J_r(\lambda))$ by the pair $(I,J^{\#})$, and
  \item
the direct sum $(J(0),I)$ of all pairs of the form $(J_r(0),I_r)$ by the pair $(J(0)^{\#},I)$,
\end{itemize}
in which $J^{\#}$ and $J(0)^{\#}$ are the Weyr matrices from Definition \ref{jyf}. We obtain a canonical matrix pair of the form
\begin{equation}       \label{3.1b}
\bigoplus_{i=1}^l(F_{p_i}^T,
G_{p_i}^T)\oplus (I,J^{\#}) \oplus
(J(0)^{\#},I)\oplus
\bigoplus_{i=1}^r(F_{q_i},
G_{q_i});
\end{equation}
in which we suppose that
\begin{equation}\label{juy}
p_1\le\dots\le p_l,\qquad
q_1\ge\dots\ge q_r.
\end{equation}
(This special ordering of direct
summands of \eqref{3.1b} admits to construct its
miniversal deformation that is block triangular.)

Denote by
\[
0^{\uparrow}:=\begin{bmatrix}
*&\cdots&*
\\
&\text{\Large 0}
\end{bmatrix},\ \
0^{\downarrow}:=\begin{bmatrix}
&\text{\Large 0}\\ *&\cdots&*
\end{bmatrix},\ \
0^{\leftarrow}:=\begin{bmatrix}
*\\\vdots&\text{\Large 0}\\ *
\end{bmatrix},\ \
0^{\rightarrow}:=\begin{bmatrix}
&*\\\text{\Large 0}&\vdots\\& *
\end{bmatrix}
\]
the
matrices, in which the entries of the first row, the last row, the first column, and the last column, respectively, are stars and the other entries are zero, and write
\[
{\cal Z}:=
\begin{bmatrix}
\begin{matrix}
*&\cdots&*
\end{matrix}&
\begin{matrix}
0&\cdots&0
\end{matrix}\\
\text{\Large 0}&
\begin{matrix}
\vdots&\ddots&\vdots\\
0&\cdots&0
\end{matrix}
\end{bmatrix}
\]
(the number of zeros in the first row of $\cal Z$ is equal to the number of rows).
The stars denote independent parameters.

In the following theorem we give a simple miniversal deformation of \eqref{3.1b} that is block triangular with respect to the partition of \eqref{3.1b} in which $J^{\#}$ and $J(0)^{\#}$ are partitioned as in Theorem \ref{t2h} and all blocks of $(F_{p_i}^T,
G_{p_i}^T)$ and $(F_{q_i},
G_{q_i})$ are $1$-by-$1$.

\begin{theorem}     \label{t3.1}
Let $(A,B)$ be a canonical matrix pair of the form \eqref{3.1b} satisfying \eqref{juy}. One of the block triangular simple miniversal deformations of $(A,B)$ has the form $({\cal A},{\cal B})$, in which
\begin{equation}\label{dnt}
{\cal A}:=\left[\begin{array}{cccc}
\begin{matrix}
F_{p_1}^T\\&F_{p_2}^T
  \\&&\ddots\\
  0&&& \multicolumn{1}{c|}{F_{p_l}^T}
\end{matrix}&&&0
          \\
          \cline{1-2}
0\vphantom{A^{A^{A}}}
&\multicolumn{1}{|c|}I&
         \\ \cline{1-3}
\begin{matrix}
0^{\rightarrow}&0^{\rightarrow}&\dots
&0^{\rightarrow}
\end{matrix} \vphantom{A^{A^{A}}} &\multicolumn{1}{|c|}0& \multicolumn{1}{c|}{J(0)^{\#} +{\cal H}^{\#}}
          \\\cline{1-3}
\begin{matrix}
0^{\rightarrow} &0^{\rightarrow}&\dots
&0^{\rightarrow}
\end{matrix}&
\multicolumn{1}{|c}{0}& \multicolumn{1}{|c|}
{\begin{matrix}
0^{\downarrow}\\ 0^{\downarrow}\\\vdots
\\0^{\downarrow}
\end{matrix}}&
\begin{matrix}\cline{1-1}
F_{q_1}\vphantom{A^{A^{A}}}\\&F_{q_2}
  \\&&\ddots\\
  0&&& F_{q_r}
\end{matrix}
\end{array}\right]
\end{equation}
and
\begin{equation}\label{dnt1}
{\cal B} :=\left[\begin{array}{cccc}
\begin{matrix}
G_{p_1}^T\\{\cal Z}^T\vphantom{A^{A^{A}}}&G_{p_2}^T
  \\\vdots&\ddots&\ddots\\
  {\cal Z}^T&\dots&{\cal Z}^T&
  \multicolumn{1}{c|}{G_{p_l}^T}
\end{matrix}&&&0\\\cline{1-2}
\begin{matrix}
0^{\leftarrow}&0^{\leftarrow} &\dots
&0^{\leftarrow}
\end{matrix}\vphantom{A^{A^{A}}}& \multicolumn{1}{|c|}{J^{\#} +{\cal K}^{\#}}&
        \\ \cline{1-3}
0\vphantom{A^{A^{A}}}
&\multicolumn{1}{|c|}0& \multicolumn{1}{c|}{I}
        \\  \cline{1-3}
\begin{matrix}
0^{\uparrow}\\ 0^{\uparrow}\\\vdots
\\0^{\uparrow}
\end{matrix}&
\multicolumn{1}{|c}
{\begin{matrix}
0^{\uparrow}\\ 0^{\uparrow}\\\vdots
\\0^{\uparrow}
\end{matrix}}& \multicolumn{1}{|c|}{0}&
\begin{matrix}
\cline{1-1}
G_{q_1}\vphantom{A^{A^{A^a}}}
\\{\cal Z}&G_{q_2}
 \\\vdots&\ddots&\ddots\\
  {\cal Z}&\dots&{\cal Z}&  G_{q_r}
\end{matrix}
\end{array}\right],
\end{equation}
where $J(0)^{\#} +{\cal H}^{\#}$ and $J^{\#} +{\cal K}^{\#}$ are the block triangular miniversal deformations \eqref{kut} and \eqref{kucj}.
\end{theorem}

\begin{proof} The following miniversal deformation of matrix pairs was obtained in \cite{gar_ser}. The matrix pair \eqref{3.1b} is equivalent to its Kronecker canonical form
\[
(A_{\text{kr}},B_{\text{kr}}):=\bigoplus_{i=1}^r(F_{q_i},
G_{q_i})\oplus (I,J) \oplus
(J(0),I)\oplus
\bigoplus_{i=1}^l(F_{p_i}^T,
G_{p_i}^T).
\]
By \cite[Theorem 4.1]{gar_ser}, one of the simple miniversal deformations of $(A_{\text{kr}},B_{\text{kr}})$ has the form $({\cal A}_{\text{kr}},{\cal B}_{\text{kr}})$, in which
\[
{\cal A}_{\text{kr}}:=\left[\begin{array}{ccc|c}
{\begin{matrix}
F_{q_r}&&&0\\&F_{q_{r-1}}
  \\&&\ddots\\
  0&&& {F_1}
\end{matrix}}&
\multicolumn{1}{|c|}{0}&
\begin{matrix}
0^{\downarrow}\\ 0^{\downarrow}\\\vdots
\\0^{\downarrow}
\end{matrix}&\begin{matrix}
0^{\rightarrow} &0^{\rightarrow}&\dots
&0^{\rightarrow}
\end{matrix}
          \\
          \hline
&\multicolumn{1}{|c|}I&0&0
          \\
         \cline{2-4}
&&\multicolumn{1}{|c|}{J(0)+{\cal H}}&\begin{matrix}
0^{\rightarrow}&0^{\rightarrow}&\dots
&0^{\rightarrow}
\end{matrix}
          \\
          \cline{3-4}
0&&&
\begin{matrix}
F_{p_l}^T\vphantom{A^{A^{A^a}}}
&&&0\\&F_{p_{l-1}}^T
  \\&&\ddots\\
  0&&& F_{p_1}^T
\end{matrix}
\end{array}\right]
\]
and
\[
{\cal B}_{\text{kr}}:=\left[\begin{array}{ccc|c}
{\begin{matrix}
G_{q_r}&{\cal Z}&\dots&{\cal Z}\\&G_{q_{r-1}}
&\ddots&\vdots
  \\&&\ddots&{\cal Z}\\
  0&&& {G_{q_{1}}}
\end{matrix}}&
\multicolumn{1}{|c|}{\begin{matrix}
0^{\uparrow}\\[1mm] 0^{\uparrow}\\[1mm]\vdots
\\[1mm]
0^{\uparrow}
\end{matrix}}&0&\begin{matrix}
0^{\uparrow}\\[1mm] 0^{\uparrow}\\[1mm]\vdots
\\[1mm]0^{\uparrow}
\end{matrix}
          \\
          \hline
&\multicolumn{1}{|c|}{J+{\cal K}}&0&\begin{matrix}
0^{\leftarrow}&0^{\leftarrow} &\dots
&0^{\leftarrow}
\end{matrix}
          \\
         \cline{2-4}
&&\multicolumn{1}{|c|}{I}&0
          \\
          \cline{3-4}
0&&&
\begin{matrix}
G_{p_l}^T\vphantom{A^{A^{A^a}}}
&{\cal Z}^T&\dots&{\cal Z}^T
\\&G_{p_{l-1}}^T&\ddots&\vdots
  \\&&\ddots&{\cal Z}^T\\
  0&&& G_{p_1}^T
\end{matrix}
\end{array}\right].
\]
In view of Theorem \ref{t2h}, the deformation $({\cal
A}_{\text{kr}},{\cal B}_{\text{kr}})$ is permutationally equivalent to the deformation
$({\cal
A},{\cal B})$ from Theorem \ref{t3.1}. (The blocks $\cal H$ and $\cal K$ in $({\cal
A}_{\text{kr}},{\cal B}_{\text{kr}})$ are lower block triangular; because of this  we reduce $({\cal
A}_{\text{kr}},{\cal B}_{\text{kr}})$ to $({\cal
A},{\cal B})$, which is lower block triangular.) \qquad
\end{proof}

\begin{remark}
Constructing $J(\lambda)^{\#}$, we for each $r$ join all $r$-by-$r$ Jordan blocks $J_r(\lambda)$ of $J(\lambda)$ in $J_r(\lambda I)$; see \eqref{kids}. We can join analogously pairs of equal sizes in \eqref{3.1b} and obtain a pair of the form
\begin{equation}\label{jyds}
\bigoplus_{i=1}^{l'}(\hat F_{p'_i}^T,
\hat G_{p'_i}^T)\oplus (I,J^{\#}) \oplus
(J(0)^{\#},I)\oplus
\bigoplus_{i=1}^{r'}(\hat F_{q'_i},
\hat G_{q'_i}),
\end{equation}
in which $
p'_1<\dots< p'_{l'}$ and
$q'_1>\dots> q'_{r'}.$
This pair is permutationally equivalent to \eqref{3.1b}. Producing the same permutations of rows and columns in \eqref{dnt} and \eqref{dnt1}, we join all $F_{p}^T,G_{p}^T, F_{q},
G_{q}$ in $\hat F_{p}^T,
\hat G_{p}^T,\hat F_{q},
\hat G_{q}$, and $0, 0^{\uparrow},
0^{\downarrow}, 0^{\leftarrow}, 0^{\rightarrow},{\cal Z}$ in $\hat 0, \hat 0^{\uparrow},
\hat 0^{\downarrow}, \hat 0^{\leftarrow}, \hat 0^{\rightarrow},\hat {\cal Z}$ which consist of blocks $0$ and $\star$ defined in \eqref{liy}; the
obtaining pair is a block triangular miniversal deformation of \eqref{jyds}.

\end{remark}

\section{Miniversal deformations of contragredient matrix pencils}
\label{s4}

Each pair of $m\times n$ and $n\times m$ matrices reduces by transformations of contragredient equivalence
\[
(A,B)\mapsto
(S^{-1}AR,R^{-1}BS),\quad\text{
$S$ and $R$ are nonsingular,}
\]
to the \emph{Dobrovol'skaya and
Ponomarev canonical form} \cite{pon} (see also \cite{hor+mer}) being a direct sum, uniquely determined up to permutation of summands, of pairs of the form
\begin{equation}       \label{3.1}
(I_r,J_r(\lambda)),\
(J_r(0),I_r),\ (F_r,G_r^T),\
(F_r^T,G_r),
\end{equation}
in which $\lambda\in {\mathbb
C}$ and the matrices $F_r$ and $G_r$ are defined in
\eqref{3.1o}.

For each matrix $M$, define the matrices
\[
{M}_{\vartriangle}:=
\begin{bmatrix}
0\,\dots\, 0\\M \\
          \end{bmatrix},\qquad
M_{\rhd}:=\begin{bmatrix}
M&\begin{matrix}0\\[-2mm]
\vdots\\[-1mm]
0
\end{matrix}
\end{bmatrix}
\]
that are obtained by adding the zero row to
the top and the zero column to the right, respectively.
Each block matrix whose blocks have the form
${\cal T}_{\vartriangle}$ (in which ${\cal T}$ is defined in \eqref{5.1}) is denoted by ${\cal H}_{\vartriangle}$. Each block matrix whose blocks have the form
${\cal T}_{\rhd}$ is denoted by ${\cal H}_{\rhd}$.

\begin{theorem}     \label{t22}
Let
\begin{equation}\label{kye}
(I,J)\oplus(A,B)
\end{equation}
be a canonical matrix pair for contragredient equivalence, in which $J$ is a nonsingular Jordan canonical matrix,
\[
(A,B):=\bigoplus_{i=1}^l(F_{p_i},
G_{p_i}^T)\oplus (I,J(0)) \oplus
(J'(0),I)\oplus
\bigoplus_{i=1}^r(F_{q_i}^T,
G_{q_i}),
\]
$J(0)$ and $J'(0)$ are Jordan matrices with the single eigenvalue $0$,
and
\[
p_1\ge p_2\ge\dots\ge p_l,\qquad q_1\le q_2\le\dots\le q_r.
\]
Then one of the simple miniversal deformations of \eqref{kye} has the form
\begin{equation}\label{hot}
(I,J+{\cal K})\oplus ({\cal A},{\cal B}),
\end{equation}
in which $J+{\cal K}$ is the deformation \eqref{kuc} of $J$ and $({\cal A},{\cal B})$ is the following deformation of $(A,B)$:
\[
{\cal A}:=\left[\begin{array}{cccc}
\begin{matrix}
F_{p_1}&{\cal T}&\dots&{\cal T}\\&F_{p_2}&\ddots
&\vdots\\&&\ddots&{\cal T}\\
  &&& F_{p_l}\\
\cline{4-4}
\end{matrix}&
\multicolumn{1}{|c|}{{\cal H}_{\vartriangle}}&
{\cal H}&
\multicolumn{1}{|c}{\cal H}
            \\
\cline{2-4}
&\multicolumn{1}{|c|}I
&{\cal H}&\multicolumn{1}{|c}
{{\cal H}_{\rhd}\vphantom{A^{A^{A}}}}
         \\ \cline{2-4}
&&\multicolumn{1}{|c|}
{J'(0)+{\cal H}}&{\cal H}\vphantom{A^{A^{A}}}\\ \cline{3-4}0
&&&\begin{matrix}
\multicolumn{1}{|c}
{G_{q_1}^T\vphantom{A^{A^{A^a}}}}
&{\cal T}&\dots&{\cal T}
\\&G_{q_2}^T
&\ddots
&\vdots
  \\&&\ddots&{\cal T}\\
 &&& G_{q_r}^T
\end{matrix}
\end{array}\right]
\]
and
\[
{\cal B} :=\left[\begin{array}{cccc}
\begin{matrix}
G_{p_1}^T+{\cal T}\mspace{-30mu}&&&
  \\
{\cal T}&\ddots
  \\
\vdots&\mspace{-30mu}
\ddots&\mspace{10mu}
  \ddots
  \\
 {\cal T}&\mspace{-30mu}\dots& \mspace{-50mu}{\cal T}&
 \multicolumn{1}{c|}{ \mspace{-30mu}G_{p_l}^T+{\cal T}}
 \end{matrix}&&&0\\\cline{1-2}
{\cal H}& \multicolumn{1}{|c|}{J(0) +{\cal H}}&\\ \cline{1-3}
\multicolumn{1}{c|}{{\cal H}_{\rhd}} &\multicolumn{1}{c|}{\cal H}    & \multicolumn{1}{c|}{I}\\\cline{1-3}
\multicolumn{1}{c|}{\cal H}&{\cal H}& \multicolumn{1}{|c|}{{\cal H}_{\vartriangle}}&
\begin{matrix}
\cline{1-2}
F_{q_1}+{\cal T}\mspace{-30mu}&&&
  \\
{\cal T}&\ddots
  \\
\vdots&\mspace{-30mu}
\ddots&\mspace{10mu}
  \ddots
  \\
 {\cal T}&\mspace{-30mu}\dots& \mspace{-50mu}{\cal T}& \mspace{-30mu}F_{q_r}+{\cal T}
 \end{matrix}\end{array}\right].
\]
\end{theorem}

\begin{proof}
The following simple miniversal deformation of \eqref{kye} was obtained in \cite[Theorem 5.1]{gar_ser}: up to obvious permutations of strips, it has the form
\begin{equation}\label{kiv}
(I,J+{\cal K})\oplus ({\cal A}',{\cal B}'),
\end{equation}
in which $J+{\cal K}$ is \eqref{kuc},
\[
{\cal A}':=\left[\begin{array}{c|c|c|c}
\begin{matrix}
F_{p_1}+{\cal T}\mspace{0mu}&{\cal T}&\dots&{\cal T}\\
&\mspace{-60mu}F_{p_2}+{\cal T}\mspace{-30mu}
&\ddots
&\vdots\\&&\ddots&{\cal T}\\
  \mspace{-30mu}0&&& \mspace{-30mu}F_{p_l}+{\cal T}\\
\end{matrix}&
0&
{\cal H}&
{\cal H}
            \\
\hline
0&I
&0&
0
         \\ \hline
{\cal H}&0&
J'(0)+{\cal H}&{\cal H}\\ \hline 0
&0&{\cal H}&\begin{matrix}
{G_{q_1}^T}\vphantom{A^{A^{A^a}}}
&{\cal T}&\dots&{\cal T}
\\&G_{q_2}^T
&\ddots
&\vdots
  \\&&\ddots&{\cal T}\\
 0&&& G_{q_r}^T
\end{matrix}
\end{array}\right],
\]
and
\[
{\cal B}' :=\left[\begin{array}{c|c|c|c}
\begin{matrix}
G_{p_1}^T&&&0
  \\
{\cal T}&G_{p_2}^T
  \\
\vdots&
\ddots&
  \ddots
  \\
 {\cal T}&\dots& {\cal T}&
 { G_{p_l}^T}
 \end{matrix}&{\cal H}&0&0\\\hline
{\cal H}& {J(0) +{\cal H}}&{\cal H}&{\cal H}\\ \hline
0 &{\cal H} & {I}&0\\\hline
{\cal H}&{\cal H}& 0&
\begin{matrix}
F_{q_1}+{\cal T}\mspace{-30mu}&&&0
  \\
{\cal T}&F_{q_2}+{\cal T}\mspace{-30mu}
  \\
\vdots&\mspace{-30mu}
\ddots&\mspace{10mu}
  \ddots
  \\
 {\cal T}&\mspace{-30mu}\dots& \mspace{-50mu}{\cal T}& \mspace{-30mu}F_{q_r}+{\cal T}
 \end{matrix}\end{array}\right];
\]

Let $(C,D)$ be the canonical pair \eqref{kye}, and let $({\cal P}, {\cal Q})$ be any matrix pair of the same size in which each entry is $0$ or $*$. By \cite[Theorem 2.1]{gar_ser},
see also the beginning of the proof of Theorem 5.1 in \cite{gar_ser}, $(C+{\cal P}, D+{\cal Q})$ is a versal (respectively, miniversal) deformation of $(C,D)$ if and only if for every pair $(M,N)$ of size of $(C,D)$  there
exist square matrices $S$ and $R$ and a pair (respectively, a unique pair) $(P, Q)$ obtained from $({\cal P},{\cal Q})$  by replacing its stars with complex numbers such that
\begin{equation}       \label{3.2}
(M,N)+(CR-SC,\,DS-RD)=(P, Q).
\end{equation}

The matrices of $(C,D)$ are block diagonal:
\[
C=C_1\oplus C_2\oplus \dots\oplus C_t,\qquad
D=D_1\oplus D_2 \oplus\dots\oplus D_t,
\]
in which $(C_i,\,D_i)$
are of the form
\eqref{3.1}. Partitioning conformally the matrices of $(M,N)$ and $({\cal P},{\cal Q})$ and equating the corresponding blocks in
\eqref{3.2}, we find that $(C+{\cal P}, D+{\cal Q})$ is a versal deformation of $(C,D)$ if and only if
\begin{equation}
\label{theo4ii}
\parbox{27em}
 {for each pair of indices $(i,j)$ and every pair $(M_{ij},N_{ij})$ of the size of $({\cal P}_{ij},{\cal Q}_{ij})$ there
exist matrices $S_{ij}$ and $R_{ij}$ and a pair $( P_{ij}, Q_{ij})$ obtained from $({\cal P}_{ij},{\cal Q}_{ij})$  by replacing its stars with complex numbers such that
\[
(M_{ij},N_{ij})+(C_iR_{ij}-S_{ij}C_j,\,
D_iS_{ij}-R_{ij}D_j)=( P_{ij},Q_{ij}).\]}
\end{equation}

Let $(C+{\cal P}', D+{\cal Q}')$ be the deformation \eqref{kiv} of $(C, D)$. Since it is versal,
\begin{equation}
\label{theog}
\parbox{27em}
 {for each pair of indices $(i,j)$ and every pair $(M_{ij},N_{ij})$ of the size of $({\cal P}'_{ij},{\cal Q}'_{ij})$ there
exist matrices $S_{ij}$ and $R_{ij}$ and a pair $(P'_{ij}, Q'_{ij})$ obtained from $({\cal P}'_{ij},{\cal Q}'_{ij})$  by replacing its stars with complex numbers such that
\[
(M_{ij},N_{ij})+ (C_iR_{ij}-S_{ij}C_j,\,
D_iS_{ij}-R_{ij}D_j)=(P'_{ij}, Q'_{ij}).\]}
\end{equation}

Let $(C+{\cal P}, D+{\cal Q})$ be the deformation \eqref{hot}. In order to prove that it is versal, let us verify the condition \eqref{theo4ii}.
If $({\cal P}_{ij},{\cal Q}_{ij})= ({\cal P}_{ij}',{\cal Q}_{ij}')$ then \eqref{theo4ii} holds by \eqref{theog}.

Let $({\cal P}_{ij},{\cal Q}_{ij})\ne ({\cal P}_{ij}',{\cal Q}_{ij}')$ for some $(i,j)$.
Since the condition \eqref{theog} holds, it suffices to verify that for each $(P_{ij}', Q_{ij}')$  obtained from $({\cal P}'_{ij},{\cal Q}'_{ij})$  by replacing its stars with complex numbers there exist matrices $S$ and $R$ and a pair $(P_{ij}, Q_{ij})$  obtained from $({\cal P}_{ij},{\cal Q}_{ij})$  by replacing its stars with complex numbers such that
\begin{equation}\label{b05}
(P_{ij}', Q_{ij}')+(C_iR-SC_j,\,
D_iS-RD_j)=( P_{ij}, Q_{ij}).
\end{equation}
The following 5 cases are possible.

\begin{description}

\item[\it Case 1: $(C_i,D_i)=(F_p,G^T_p)$ and $i=j$.]
${}$
Then
\[
(P_{ii}', Q_{ii}')=(T,\,0)=\left(
\begin{bmatrix}
&\text{\Large 0}\\ \alpha _1&\cdots&\alpha _{p-1}
\end{bmatrix},\, 0
\right)
\]
(we denote by $T$ any matrix obtained from ${\cal T}$ by replacing its stars with complex numbers).
Taking
\[
S:=\begin{bmatrix}
0&&&&0\\ \alpha _{p-1}&\ddots\\
\ddots&\ddots&\ddots\\
\alpha _2&\ddots&\ddots&\ddots\\
\alpha _1&\alpha _2&\ddots&\alpha _{p-1}&0
\end{bmatrix},\qquad
R:=\begin{bmatrix}
0&&&&0\\ \alpha _{p-1}&\ddots\\
\ddots&\ddots&\ddots\\
\alpha _3&\ddots&\ddots&\ddots\\
\alpha _2&\alpha _3&\ddots&\alpha _{p-1}&0
\end{bmatrix}
\]
in \eqref{b05}, we obtain
\[
(P_{ii},Q_{ii})=\left(0,
\begin{bmatrix}
\alpha _{p-1}\\
\vdots&\text{\Large 0}\\ \alpha _1
\end{bmatrix}
\right)=(0,T).
\]

\item[\it Case 2: $(C_i,D_i)=(F_p,G^T_p)$ and $(C_j,D_j)=(I_m,J_m(0))$.]
${}$
Then
$
(P_{ij}', Q_{ij}')=(0,T).
$
Taking $S:=-T_{\vartriangle}$ and $R:=0$ in \eqref{b05}, we obtain
$
(P_{ij}, Q_{ij})=(T_{\vartriangle},0).
$

\item[\it Case 3: $(C_i,D_i)=(I_m,J_m(0))$ and $(C_j,D_j)=(J_n(0),I_n)$.]
Then
$
(P_{ij}', Q_{ij}')=(0,T).
$
Taking $S:=0$ and $R:=T$ in \eqref{b05}, we obtain
$
(P_{ij}, Q_{ij})=(T,0).
$

\item[\it Case 4: $(C_i,D_i)=(I_m,J_m(0))$ and $(G_j,D_j)=(G_q^T,F_q)$.]
${}$
Then
$
(P_{ij}', Q_{ij}')=(0,T).
$
Taking $S:=0$ and $R:=T_{\rhd}$ in \eqref{b05}, we obtain
$
(P_{ij}, Q_{ij})=(T_{\rhd},0).
$

\item[\it Case 5: $(C_i,D_i)=(J_n(0),I_n)$ and $(G_j,D_j)=(F_p,G_p^T,)$.]
${}$
Then
$
(P_{ij}', Q_{ij}')=(T,0).
$
Taking $S:=T_{\rhd}$ and $R:=0$ in \eqref{b05}, we obtain
$
(P_{ij}, Q_{ij})=(0,T_{\rhd}).
$
\end{description}

We have proved that the deformation \eqref{hot} is versal. It is miniversal since it has the same number of parameters as the miniversal deformation \eqref{kiv}.\qquad
\end{proof}

\begin{remark}
The deformation $(I,J+{\cal K})\oplus ({\cal A},{\cal B})$ from Theorem {\rm\ref{t22}} can be made block triangular by the following permutations of its rows and columns, which are transformations of contragredient equivalence:
\begin{itemize}
  \item
First, we reduce $(I,J+{\cal K})$ to the form $(I,J^{\#} +{\cal K}^{\#})$, in which $J^{\#} +{\cal K}^{\#}$ is defined in \eqref{kucj}.

  \item
Second, we reduce the diagonal block $J(0)+{\cal H}$ in $\cal B$ to the form
$J(0)^{\#}+{\cal H}^{\#}$
$($defined in \eqref{kut}$)$ by the permutations of rows and columns of $\cal B$ described in Definition {\rm\ref{jyf}}. Then we make the contragredient permutations of rows and columns of $\cal A$.

  \item
Finally, we reduce the diagonal block $J'(0)+{\cal H}$ in $\cal A$ to the form
$J'(0)^{\#}+{\cal H}^{\#}$
$($defined in \eqref{kut}$)$ by the permutations of rows and columns of $\cal A$ described in Definition {\rm\ref{jyf}}, and make the contragredient permutations of rows and columns of $\cal B$. The obtained deformation $J'(0)^{\#}+{\cal H}^{\#}$ is lower block triangular, we make it upper block triangular by transformations
\[
P(J'(0)^{\#}+{\cal H}^{\#})P,\qquad P:=\begin{bmatrix}0&&1\\
&\ddd\\1&&0\end{bmatrix}
\]
$($i.e., we rearrange in the inverse order the rows and columns of $\cal A$ that cross $J'(0)^{\#}+{\cal H}^{\#}$ and make the contragredient permutations of rows and columns of $\cal B$$)$.
\end{itemize}
\end{remark}

\end{document}